%% file: higher-derivatives.tex
\documentclass{article}

\input{preamble}

\title{Sobolev bounds and counterexamples for the second derivative of the maximal function in one dimension}
\author{Julian Weigt\footnote{
University of Warwick,
Mathematics Institute,
Coventry CV47AL,
United Kingdom,
\texttt{julian.weigt@warwick.ac.uk}
}
}

\begin{document}

\maketitle

\begin{abstract}
We investigate the question whether the \(L^1(\mathbb{R})\)-norm of the second derivative of the uncentered Hardy-Littlewood maximal function can be bounded by a constant times the \(L^1(\mathbb{R})\)-norm of the function itself.
We give a positive answer for a class of functions that contains Sobolev functions on the real line which are decreasing away from the origin and even,
and we provide a counterexample which is also decreasing away from the origin but not even.
\end{abstract}

\begingroup
\begin{NoHyper}%
\renewcommand\thefootnote{}\footnotetext{%
2020 \textit{Mathematics Subject Classification.} 42B25, 26A45.\\%
\textit{Key words and phrases.} maximal function, second derivative.\\%
This work was supported by the European Union's Horizon 2020 research and innovation programme (Grant agreement No. 948021).%
}%
\addtocounter{footnote}{-1}%
\end{NoHyper}%
\endgroup

\section{Introduction}

\subsection{Background}

\subsubsection{The original questions}
We are studying the second derivative of the uncentered Hardy-Littlewood maximal function which for a function \(f:\mathbb{R}\rightarrow\mathbb{R}\) is given by
\[
\M f(x)
=
\sup_{a<x<b}
\f1{b-a}
\int_a^b
f(y)
\intd y
.
\]
The maximal function of a function on \(\mathbb{R}^d\) is defined using averages over balls instead of intervals.
Usually the maximal function is defined in terms of averages over the absolute value \(|f(y)|\) instead of \(f(y)\), but for the purposes of studying its regularity the formulation above without the absolute value turned out to be more appropriate.

The Hardy-Littlewood maximal function theorem states that for \(p>1\) the bound
\begin{equation}
\label{eq_hardylittlewood}
\|\M f\|_{L^p(\mathbb{R}^d)}
\leq C_{p,d}
\|f\|_{L^p(\mathbb{R}^d)}
\end{equation}
holds, and for \(p=1\) we have the weak version
\begin{equation}
\label{eq_hardylittlewoodweak}
\|\M f\|_{L^{1,\infty}(\mathbb{R}^d)}
\leq C_{1,d}
\|f\|_{L^1(\mathbb{R}^d)}
.
\end{equation}
In 1997 Juha Kinnunen proved in \cite{Kin97} the derivative version
\begin{equation}
\label{eq_kinnunen}
\|\nabla\M f\|_{L^p(\mathbb{R}^d)}
\leq C_{p,d}
\|\nabla f\|_{L^p(\mathbb{R}^d)}
\end{equation}
if \(p>1\).
In 2002 Tanaka showed in \cite{MR1898539} that for \(d=1\) also
\begin{equation}
\label{eq_endpoint}
\|\nabla\M f\|_{L^1(\mathbb{R}^d)}
\leq C_{1,d}
\|\nabla f\|_{L^1(\mathbb{R}^d)}
\end{equation}
holds, and in 2004 Haj\l{}asz and Onninen formally asked in \cite{MR2041705} the question if this is the case also in higher dimensions and for other maximal operators such as the centered Hardy-Littlewood maximal operator \(\Mc\).

\subsubsection{The progress so far}
Much research has since then been dedicated to proving \cref{eq_endpoint} in higher dimensions,
successfully only for specific maximal operators and in specific cases.
In particular its original formulations for the centered and uncentered Hardy-Littlewood maximal function remain yet to be answered.
Moreover, there has been an increased interest in related operators such as fractional maximal operators and convolution operators,
and in related questions such as the operator continuity of \(f\mapsto\nabla\M f\).
See \cite{carneiro2019regularity} for a survey of the field and \cite{gonzálezriquelme2021continuity,weigt2024variation,zbMATH07780902} for more recent results among many many others.
While \cref{eq_hardylittlewood,eq_kinnunen} are equivalent to prove for a range of maximal operators, \cref{eq_endpoint} and its related questions turned out to be generally more difficult for centered than for uncentered maximal operators.

For the purpose of this work we now restrict our attention to the classical centered and uncentered Hardy-Littlewood maximal operators in one dimension, where \cref{eq_endpoint} has already been proven.
For the uncentered Hardy-Littlewood maximal operator it is even known due to Aldaz and P\'erez L\'azaro, \cite{MR2276629}, that the best constant \(C_{1,1}\) in \cref{eq_endpoint} is \(1\).
\Cref{eq_endpoint} was proven for the centered Hardy-Littlewood maximal operator in \cite{MR3310075} with significantly more effort, but the corresponding best constant \(C_{1,1}\) remains unknown.
It is conjectured to be \(1\) as well, which has been confirmed in \cite{bilz2021onedimensional} at least within the class of characteristic functions.
The discrete setting, i.e.\ when considering functions \(f:\mathbb{N}\rightarrow\mathbb{R}\), has in most instances so far been a mirror of the continuous setting \cite{BCHP12,CH12,Mad17,bilz2021onedimensional}.

\subsubsection{Extremizers}
\label{subsubsec_extremizers}
It appears that extremizers and extremizing sequences of most of the previously discussed bounds can be found among radially decreasing functions.
For example for the uncentered Hardy-Littlewood maximal function in one dimension this is known to be true in two instances.
In \cite{zbMATH00936588} Grafakos and Montgomery-Smith show that for every \(p>1\) truncations of \(x\mapsto |x|^{-1/p}\) maximize the constant in \cref{eq_hardylittlewood} in this case,
and it follows from \cite{MR2276629}  that all functions that monotonously decrease to zero from the origin are extremizers of \cref{eq_endpoint} and of its relaxed version
\begin{equation}
\label{eq_variationbound}
\var(\M f)
\leq C
\var(f)
.
\end{equation}
For the centered Hardy-Littlewood maximal operator in one dimension the characteristic function of an interval is the only extremizer of \cref{eq_variationbound} among all characteristic functions \cite{bilz2021onedimensional}, and if the best constant also for general functions is indeed 1 then this is also a general extremizer.

The only known exception to this pattern is \cref{eq_hardylittlewoodweak} for the centered maximal function in one dimension.
In \cite{zbMATH02057955} Melas constructed a very involved extremizing sequence which in particular beats the largest constant attainable by radially decreasing functions.
This is indicated by the fact that already \(\|\Mc(\ind{(-2,-1)\cup(1,2)})\|_{L^{1,\infty}(\mathbb{R})}>\|\Mc(\ind{(-1,1)})\|_{L^{1,\infty}(\mathbb{R})}\).
To the best of my knowledge for no other of the bounds discussed so far or no other maximal function exists a proof that the best constant is not achieved by radially decreasing functions.
However also this setting of \(L^p\)-bounds for \(\Mc\) remains characterized by radially decreasing functions to the extent that for any \(1\leq p\leq\infty\) the bound \cref{eq_hardylittlewood} holds if and only if it holds for radially decreasing functions.

\subsection{Definition of the variation}
\label{subsec_definitions}

We roughly follow \cite[Section~5]{MR3409135}.
For an open set \(U\subset\mathbb{R}\) the variation of a locally integrable function \(f\in L^1_\loc(U)\) is defined as
\[
	\var_U(f)
	=
	\sup\biggl\{
		-
		\int_U f\varphi'
		:
		\varphi\in C^1_{\tx c}(U)
		,\ 
		|\varphi|\leq1
	\biggr\}
	.
\]
Denote by \(U_f\) the set of points of approximate continuity of \(f\) in \(U\).
If \(U\) is an open interval then for any \(1\leq q<\infty\) define the \(q\)-variation by
\[
	\var_U^q(f)
	=
	\sup\biggl\{
		\sum_{i=1}^m
		\bigl|
		f(x_{i+1})
		-
		f(x_i)
		\bigr|^q
		:
		m\in\mathbb{N}
		,\ 
		x_i\in U_f
		,\ 
		x_1<\ldots<x_m
	\biggr\}^{\f1q}
\]
and
\[
	\var_U^\infty(f)
	=
	\sup\biggl\{
		\bigl|
		f(x)
		-
		f(y)
		\bigr|
		:
		x,y\in U_f
	\biggr\}
	.
\]
We have \(\var_U^1(f)=\var_U(f)\), and if it is finite then there exists a Radon measure denoted by \(f'\) such that for all \(\varphi\in C^1_{\tx c}(U)\) we have
\[
	\int_U\varphi\intd(f')
	=
	-\int_Uf\varphi'
	,
\]
and moreover for the total variation measure \(|f'|\) we have
\[
	|f'|(U)
	=
	\var_U(f)
	.
\]

The function \(f\) is Sobolev if and only if \(f'\) is absolutely continuous with respect to Lebesgue measure, and it this case it equals its weak derivative.
That means \cref{eq_variationbound} makes sense for all functions \(f\) that have a vector valued Radon measure as its gradient and are not necessarily Sobolev in contrast to \cref{eq_endpoint}.
Examples for such functions are characteristic functions of sets with finite perimeter.
If \(f\) is Sobolev then it has a continuous representative \(\tilde f\) which means that if \(U\) is an open interval we have
\[
	\var_U^q(f)
	=
	\sup\biggl\{
		\sum_{i=1}^m
		\bigl|
		\tilde f(x_{i+1})
		-
		\tilde f(x_i)
		\bigr|^q
		:
		m\in\mathbb{N}
		,\ 
		x_i\in U
		,\ 
		x_1<\ldots<x_m
	\biggr\}^{\f1q}
	.
\]
Many proofs of regularity of maximal functions in one dimension rely on this formula, in particular in this manuscript.

\subsection{Second derivatives}

\subsubsection{The uncentered maximal operator}
Several aspects drastically differ for second derivative bounds of maximal functions when compared to the previously discussed zeroth and first derivative bounds.
Given a smooth, even and radially decreasing bump function \(g\) supported on the interval \((-1,1)\) its maximal function is strictly increasing on \((-\infty,0)\) and decreasing on \((0,\infty)\).
Moreover, for \(f(x)=g(x+2)+g(x-2)\) we have for \(-1\leq x\leq 0\) that \(\M f(x)=\M g(x+2)\) and for \(0\leq x\leq 1\) we have \(\M f(x)=\M g(x-2)\).
This means that in \(x=0\) the derivative of \(\M f\) jumps from \(\M g'(2)<0\) to \(\M g'(-2)=-\M g'(2)>0\).
This means \(\M f\) is not twice weakly differentiable and hence \(\M f''\not\in L^p(\mathbb{R})\) for any \(p\geq 1\).
The function \(f\) on the other hand is smooth and compactly supported and thus \(f''\in L^p(\mathbb{R})\), which means that
\begin{equation}
	\label{eq_ddmfboundp}
	\|(\M f)''\|_p
	\leq C_p
	\|f''\|_p
\end{equation}
can \emph{not} hold in any sense if \(p>1\).

Like $\M f$, also $|f|$ can have kinks even if $f$ is a smooth function on $\mathbb{R}$.
However, in the endpoint $p=1$ it is straightforward to see that \(\var(|f|')\leq2\var(f')\) holds for $f\in C^1_0(\mathbb{R})$.
This begs the question if also \((\M f)'\) has finite variation, i.e.\ if \((\M f)''\) still exist as a Radon measure.
One may argue though that the similarity between $\M$ and $f\mapsto|f|$ is limited here because we define $\M$ without the absolute value.

In any case, the previous double bump example is unlikely to contradict \cref{eq_ddmfboundp} for \(p=1\) in its relaxed form
\begin{equation}
\label{eq_ddmf1}
	\var((\M f)')
	\leq C
	\var(f')
	.
\end{equation}
Indeed, in \cite{temur2022second} Faruk Temur has proven \cref{eq_ddmf1} for characteristic functions in the discrete setting.
Our first main result is that in the continuous setting counterexamples against \cref{eq_ddmf1} exist and they even violate the corresponding weak bound.

\begin{theorem}
\label{theo_counterexample}
For every \(1\leq q<\infty\) and \(C>0\) there exists a weakly differentiable function \(f:\mathbb{R}\rightarrow\mathbb{R}\) with only one local maximum such that
\[
	\sup_{\lambda>0}
	\lambda
	\lm{\{x\in\mathbb{R}:|(\M f)''(x)|>\lambda\}}
	> C
	\var(f')
\]
and
\[
	\var^q((\M f)')
	> C
	\var^q(f')
.
\]
\end{theorem}

In the endpoint \(q=\infty\) the bound
\begin{equation}
\label{eq_infinityvariation}
	\var^\infty((\M f)')
	\leq
	\var^\infty(f')
\end{equation}
holds for the following reason.
Recall Luiro's formula, originally \cite[Theorem~3.1]{zbMATH05120227}, according to which in almost every \(x\) we have
\[
	(\M f)'(x)
	=
	\f1{b_x-a_x}
	\int_{a_x}^{b_x}
	f'
\]
with \(a_x\leq x\leq b_x\) so that
\[
	\M f(x)
	=
	\f1{b_x-a_x}
	\int_{a_x}^{b_x}
	f
	,
\]
with appropriate modifications if \(a_x=b_x\) or \(b_x-a_x=\infty\).
We can conclude that for every \(x\in\mathbb{R}_{(\M f)'}\) there are \(x_0,x_1\in\mathbb{R}_{f'}\) with \(f'(x_0)\leq(\M f)'(x)\leq f'(x_1)\) and thus
\begin{align*}
	\var^\infty((\M f)')
	&=
	\esssup((\M f)')-\essinf((\M f)')
	\\
	&\leq
	\esssup(f')-\essinf(f')
	=
	\var^\infty(f')
	.
\end{align*}
It is unknown if there exists a \(C<1\) such that \(\var^\infty((\M f'))\leq C\var^\infty(f')\), similarly to how it is unknown if \cref{eq_kinnunen} holds for \(p=\infty\) with some \(C_{\infty,d}<1\).

The proof of \cref{theo_counterexample} exploits a different phenomenon than the double bump counterexample \(p>1\) from above.
The starting point is that for \(A,B>0\) with \(A\) much smaller than \(B\) the function
\begin{equation}
\label{eq_hatfunction}
f(x)
=
\begin{cases}
Ax&x\leq0\\
-Bx&x\geq0
\end{cases}
\end{equation}
satisfies \((\M f)'(x)\approx-\sqrt{AB}\).
This means disturbing \(f'\) by \(\varepsilon A\) for \(x\leq0\) leads to a disturbance of \((\M f)'\) by \(\sim\varepsilon\sqrt{AB}\) for \(x\geq0\).
Note, that \(A\ll B\) implies \(\varepsilon\sqrt{AB}\gg\varepsilon A\).
Than means a small variation of \(f'\) can lead to a much larger variation of \((\M f)'\), contradicting a variation bound for the first derivative.
We formally prove \cref{theo_counterexample} in \cref{sec_counterexample}.

By discretization \cref{theo_counterexample} can be turned into a counterexample in the discrete setting, which means that the result by \cite{temur2022second} cannot be generalized from characteristic functions to all functions.
However, if we require the functions to be in addition even, which makes them radially decreasing, then \cref{eq_ddmf1} holds.

\begin{theorem}
\label{theorem_radial}
There exists a \(C>0\) such that \cref{eq_ddmf1} holds for all \(f:\mathbb{R}\rightarrow\mathbb{R}\) which are even and nonincreasing on \((0,\infty)\).
\end{theorem}

This is a bit surprising in view of \cref{subsubsec_extremizers}.
The zeroth and first derivative bounds \cref{eq_hardylittlewood,eq_hardylittlewoodweak,eq_kinnunen,eq_variationbound,eq_endpoint} hold or are conjectured to hold if and only if they hold for radially decreasing functions.
For the uncentered maximal function radially decreasing functions are even known or conjectured to be their extremizers, and even for the centered maximal function this may hold in some cases.
\Cref{theo_counterexample,theorem_radial} show that the second derivative bound \cref{eq_ddmf1} does not obey this pattern, already for the uncentered maximal function.

The same may be the case for \cref{eq_ddmfboundp}.
The counterexample for \(p>1\) uses that for a smooth function \(g\) its maximal function \(\M g\) can have corners, but it is unclear if this can occur also if \(g\) is even and decreasing on \((0,\infty)\), or more generally if it has only one local maximum like in \cref{theo_counterexample}.
Conversely, the functions \(f_1,f_2,\ldots\) constructed in \cref{theo_counterexample} all have corners, which means that their second derivative does not have a finite \(L^p(\mathbb{R})\)-norm for \(p>1\).
This suggests the following \lcnamecref{que_radiallydecreasingp}.

\begin{question}
\label{que_radiallydecreasingp}
Does \cref{eq_ddmfboundp} hold for some \(p>1\) for even functions \(f:\mathbb{R}\rightarrow\mathbb{R}\) which are nonincreasing on \((0,\infty)\)?
Does it hold for functions \(f\) which have at most one local maximum?
\end{question}

Our proof of \cref{theorem_radial} relies on the following generalization.

\begin{theorem}
\label{theorem_general}
For every \(c>0\) there exists a \(C>0\) such that for every \(f:\mathbb{R}\rightarrow\mathbb{R}\) which is nonincreasing on \((0,\infty)\), nondecreasing on \((-\infty,0)\) and which for every \(x>0\) satisfies
\begin{equation}
\label{eq_intervalsrestricted}
\M f(x)
=
\sup_{y\geq -cx}
\f1{x-y}
\int_y^x f
\end{equation}
we have \cref{eq_ddmf1}, i.e.
\[
\var((\M f)')
\leq C
\var(f')
.
\]
\end{theorem}

\begin{proof}[Proof of \cref{theorem_radial}]
Let \(x>0\) and \(y<-x\).
Then for any \(y<z<-x<u<x\) we have \(f(z)\leq f(u)\) which means
\[
\f1{-y-x}
\int_y^{-x} f
\leq
\f1{2x}
\int_{-x}^x f
,
\]
and thus
\begin{align*}
\f1{x-y}
\int_y^x f
&=
\f{-y-x}{x-y}
\f1{-y-x}
\int_y^{-x} f
+
\f{2x}{x-y}
\f1{2x}
\int_{-x}^x f
\\
&\leq
\max\Bigl\{
\f1{-y-x}
\int_y^{-x} f
,
\f1{2x}
\int_{-x}^x f
\Bigr\}
\\
&=
\f1{2x}
\int_{-x}^x f
.
\end{align*}
Thus \cref{eq_intervalsrestricted} holds with \(c=1\) and we can apply \cref{theorem_general}.
\end{proof}

We prove \cref{theorem_general} in \cref{section_proof} for functions with a lower and upper Lipschitz bound, and then extend the result to full generality in \cref{sec_approximation} by approximation.

\subsubsection{The centered maximal operator \texorpdfstring{\(\Mc\)}{Mc}}

To disprove \cref{eq_ddmfboundp} for \(p>1\) a similar argument as in the uncentered case applies for the centered Hardy-Littlewood maximal function \(\Mc f\).
Also the proof of \cref{eq_infinityvariation} works for the centered maximal function \(\Mc f\) the same way as it does for the uncentered.

The validity of \cref{theo_counterexample,theorem_radial,theorem_general} remains open for the centered maximal operator \(\Mc\).
Since the function \cref{eq_hatfunction} is concave, its centered maximal function \(\Mc f\) is equal to \(f\) which means that this particular strategy to obtain a counterexample for \cref{theo_counterexample} does not work.
Similarly, for a radially decreasing function its centered maximal function does not satisfy a lower bound on the radii it averages over similar to \cref{eq_intervalsrestricted} in \cref{theorem_general}, which means that our proof of \cref{theorem_radial} fails for the centered maximal function.
That means even if a proof of \cref{theorem_general} is possible also for the centered maximal function, i.e.\ for functions satisfying
\[
\Mc f(x)
=
\sup_{r>(1+c)|x|}
\f1{2r}
\int_{x-r}^{x+r}
f
\]
instead of \cref{eq_intervalsrestricted}, it is not clear how interesting this class of functions is.

\begin{question}
Do \cref{eq_ddmf1,theo_counterexample,theorem_general,theorem_radial,que_radiallydecreasingp} hold for the centered Hardy-Littlewood maximal function?
\end{question}

\section{Proof assuming a lower Lipschitz bound}
\label{section_proof}

In this \lcnamecref{section_proof} we prove \cref{theorem_smooth}, which is \cref{theorem_general} for functions with certain regularity.
For the rest of this \lcnamecref{section_proof} let \(K>0\) and let \(f:\mathbb{R}\rightarrow\mathbb{R}\) such that for all \(x<y\) with \(0\not\in(x,y)\) we have
\begin{equation}
\label{eq_diffboundassumption}
\f1K
\leq
-\sign(x+y)
\f{f(y)-f(x)}{y-x}
\leq
K
.
\end{equation}
We assume that \cref{eq_diffboundassumption} holds throughout this \lcnamecref{section_proof}.
For any \(y\neq x\) abbreviate
\[
A_f(y,x)
=
\f1{x-y}
\int_y^x
f
,
\]
so that
\[
	\M f(x)
	=
	\sup_{y\neq x}
	A_f(y,x)
	.
\]
By \cref{eq_diffboundassumption} the function \(f\) is continuous, strictly increasing on \((-\infty,0)\), strictly decreasing on \((0,\infty)\) and moreover \(\lim_{|x|\rightarrow\infty}f(x)=-\infty\).
We can infer that \(\M f(0)=f(0)\), that \(\M f\) is continuous in \(0\),
and that for every \(\pm x>0\) there exists is a unique value \(\pm a(x)<0\) for which
\[
\M f(x)
=
A_f(a(x),x)
.
\]
Since \(f\) is continuous, the map \(A_f\) is continuously differentiable in \((x,a(x))\) and thus 
\[
0
=
\partial_1 A_f(a(x),x)
=
\f1{(x-a(x))^2}
\int_a^x f
-
\f{f(a(x))}{x-a(x)}
=
\f{\M f(x)-f(a(x))}{x-a(x)}
.
\]
Since the restriction \(f_\pm\seq f|_{\{y:\pm y>0\}}\) is invertible we can conclude
\begin{equation}
\label{eq_a}
a(x)
=
(f_\mp)^{-1}(\M f(x))
.
\end{equation}

\begin{lemma}
\label{lemma_lipschitz}
The maximal function \(\M f\) is \(K\)-Lipschitz and \(a\) is \(K^2\)-Lipschitz.
\end{lemma}

\begin{proof}
By continuity of \(\M f\) in \(0\) and by symmetry it suffices to consider \(0<x<y\).
Then \(\M f(y)\leq\M f(x)\) and \(a(y)\leq a(x)\), and moreover for every \(x<t<y\) by \cref{eq_diffboundassumption} we have
\[
	f(t)
	\geq
	f(y)
	\geq
	f(0)
	-
	Ky
	\geq
	f(a(x))
	-
	K(y-a(x))
.
\]
Inserting \(f(a(x))=\M f(x)\) we obtain
\begin{align*}
\M f(y)
&\geq
\f1{y-a(x)}
\int_{a(x)}^y
f
\\
&\geq
\f1{y-a(x)}
\Bigl(
\int_{a(x)}^x
f
+
(y-x)
[
\M f(x)
-
K(y-a(x))
]
\Bigr)
\\
&=
\M f(x)
-
K(y-x)
,
\end{align*}
which means that \(\M f\) is \(K\)-Lipschitz.
By \cref{eq_diffboundassumption} the restrictions \(f_\pm\) are \(1/K\)-lower Lipschitz and we consequently obtain from \cref{eq_a} that \(a\) is a composition of two \(K\)-Lipschitz functions and thus itself \(K^2\)-Lipschitz.
\end{proof}

\begin{lemma}
\label{lem_mandadifferentiable}
For any \(n=0,1,\ldots\) if \(f\) is \(n\) times differentiable in \(x\neq0\) and in \(a(x)\) then \(\M f\) is \(n+1\) times differentiable in \(x\) and \(a\) is \(n\) times differentiable in \(x\).
In particular,
\begin{align}
\label{eq_diffmf}
	(\M f)'(x)
	&=
	\f{f(x)-\M f(x)}{x-a(x)}
	,\\
	\label{eq_da}
	a'(x)
	&=
	\f{(\M f)'(x)}{f'(a(x))}
	,\\
	\label{eq_ddmf}
	(\M f)''(x)
	&=
	\f
	{
		f'(x)
		-
		(2-a'(x))(\M f)'(x)
	}{x-a(x)}
\end{align}
\end{lemma}

\begin{proof}
Since \(\partial_1 A_f(a(x),x)=0\) and since \(a(x)\) is \(K^2\)-Lipschitz by \cref{lemma_lipschitz} we have for \(y\rightarrow x\) that
\[
\f{|
A_f(a(x),x)-A_f(a(y),x)
|}{|
x-y
|}
=
\underbrace{
\f{|
A_f(a(x),x)-A_f((a(y),x)
|}{|
a(x)-a(y)
|}
}_{\rightarrow0}
\underbrace{
\f{|
a(x)-a(y)
|}{|
x-y
|}
}_{\leq K^2}
\to0
\]
and therefore
\begin{align*}
	(\M f)'(x)
&=
\partial_2 A_f(a(x),x)
=
\f\intd{\intd x}
\Bigl(
\f1{x-y}
\int_y^x
f
\Bigr)
\Bigr|_{y=a(x)}
\\
&=
-
\f1{(x-a(x))^2}
\int_{a(x)}^x
f
+
\f{f(x)}{x-a(x)}
\end{align*}
which equals \cref{eq_diffmf}.
If \(f\) is differentiable in \(x\) and \(a(x)\) then \cref{eq_da} follows as a consequence of \cref{eq_a}, and from \cref{eq_diffmf} we obtain
\[
	(\M f)''(x)
	=
	\f{f'(x)-(\M f)'(x)}{x-a(x)}
	-
	\f{1-a'(x)}{x-a(x)}(\M f)'(x)
\]
which equals \cref{eq_ddmf}.
It follows by induction that for every \(n\in\mathbb{N}\) the \((n+1)\)th derivative \((\M f)^{(n+1)}(x)\) is a smooth function of
\(
f^{(i)}(x)
,
(\M f)^{(i)}(x)
\)
and
\(
a^{(i)}(x)
\)
for \(i=0,\ldots,n\),
of
\(f^{(i)}(a(x))\)
for \(i=1,\ldots,n\) and of \(x\).
Moreover, it follows that \(a^{(n)}(x)\) is a smooth function of
\(
f^{(i)}(a(x))
\)
and
\(
(\M f)^{(i)}(x)
\)
for \(i=0,\ldots,n\)
and of
\(a^{(i)}(x)\)
for \(i=1,\ldots,n-1\).
\end{proof}

From now on and for the rest of this \lcnamecref{section_proof} we assume in addition to \cref{eq_diffboundassumption} that \(f\) is twice differentiable on \(\mathbb{R}\setminus\{0\}\).
That means \cref{eq_diffmf,eq_da,eq_ddmf} hold and \((\M f)''\) is continuous on \(\mathbb{R}\setminus\{0\}\).
Moreover, we let \(c>0\) and assume that for any \(x>0\) we have
\begin{equation}
\label{eq_abound}
-a(x)\leq cx
.
\end{equation}

\begin{lemma}
\label{lem_boundbetweenzeronum}
Let \(D>0\) and for \(i=0,1\) let \(u_i,v_i,w_i>0\) be real numbers such that
\[
u_iv_i
=
w_i^2
+
2w_iv_i
\]
and
\(w_i\leq D v_i\).
Then
\[
|w_1-w_0|
\leq2(2+D)^2
\bigl(
|u_1-u_0|
+
|v_1-v_0|
\bigr)
.
\]
\end{lemma}

\begin{proof}
By symmetry it suffices to consider the case \(v_0\leq v_1\).
First, note that
\begin{align}
\label{eq_ddmfnum}
&&
u_i
&=
(2+w_i/v_i)w_i
\\
\nonumber
\iff&&
u_iv_i
&=
w_i^2
+
2w_iv_i
\\
\nonumber
\iff&&
v_i(u_i+v_i)
&=
(w_i+v_i)^2
\\
\label{eq_ddmfeqzeronum}
\iff&&
w_i
&=
v_i\bigl(\sqrt{1+u_i/v_i}-1\bigr)
.
\end{align}
By \cref{eq_ddmfeqzeronum} we have
\begin{align}
\nonumber
|w_1-w_0|
&\leq
|v_0-v_1|
+
\Bigl|
v_1\sqrt{1+u_1/v_1}
-
v_0\sqrt{1+u_0/v_0}
\Bigr|
\\
\label{eq_mdiffsplit}
&\leq
|v_1-v_0|
\Bigl(
1
+
\max_{i=0,1}
\sqrt{1+u_i/v_i}
\Bigr)
+
v_1
\Bigl|
\sqrt{1+u_1/v_1}
-
\sqrt{1+u_0/v_0}
\Bigr|
\end{align}
and
\[
\sqrt{1+u_i/v_i}
=
w_i/v_i
+
1
\leq
1
+
D
\]
and by \cref{eq_ddmfnum} we have
\begin{equation}
\label{eq_fleqg}
u_i
\leq
(2+D)
w_i
\leq
(2+D)D
v_i
.
\end{equation}
It suffices to consider the case
\[
|v_1-v_0|
\leq
\f{
|w_1-w_0|
}{
2(2+D)
}
\]
which implies
\begin{equation}
\label{eq_gtermsmall}
|v_1-v_0|
\Bigl(
1
+
\max_{i=0,1}
\sqrt{1+u_i/v_i}
\Bigr)
\leq
\f{
|w_1-w_0|
}2
.
\end{equation}
Using \cref{eq_mdiffsplit,eq_fleqg,eq_gtermsmall} we can conclude
\begin{align*}
\f{
|w_1-w_0|
}2
&\leq
v_1
\Bigl|
\sqrt{1+u_1/v_1}
-
\sqrt{1+u_0/v_0}
\Bigr|
\\
&\leq
v_1
\Bigl|
\sqrt{1+u_1/v_1}
-
\sqrt{1+u_0/v_0}
\Bigr|
\Bigl(
\sqrt{1+u_0/v_0}
+
\sqrt{1+u_1/v_1}
\Bigr)
\\
&=
\Bigl|
u_1
-
\f{
u_0v_1
}{
v_0
}
\Bigr|
\\
&=
\Bigl|
u_1-u_0
+
\f{u_0(v_0-v_1)}{v_0}
\Bigr|
\\
&\leq
|u_1-u_0|
+
(2+D)D
|v_1-v_0|
,
\end{align*}
finishing the proof.
\end{proof}

\begin{remark}
\Cref{lem_boundbetweenzeronum} fails without the condition \(w_i\leq Dv_i\) by the following example.
Let \(D\) be large, \(u_i=D\), \(v_0=1/D\) and \(v_1=2/D\).
Since
\[
\sqrt{1+t}
=
\sqrt t
+
\ord(t^{-1/2})
\]
by \cref{eq_ddmfeqzeronum} we have
\[
w_i
=
v_i
(
D
+
\ord(1)
)
\]
and thus
\[
|w_1-w_0|
=
1
+
\ord(D^{-1})
\sim
D
(
|u_1-u_0|
+
|v_1-v_0|
)
.
\]
This shows that the rate in \(D\) in \cref{lem_boundbetweenzeronum} must be at least \(D^1\).
\end{remark}

\begin{proposition}
\label{proposition_lowspeed}
	Let \(0<x_0<x_1<\infty\) be such that for \(i=0,1\) we have \((\M f)''(x_i)=0\) and \(-a'(x_i)\leq D\).
Then
\[
	|(\M f)'(x_0)-(\M f)'(x_1)|
\leq2(2+D)^2
\bigl(
|f'(x_1)-f'(x_0)|
+
|f'(a(x_1))-f'(a(x_0))|
\bigr)
.
\]
\end{proposition}

\begin{proof}
Let \(w_i=-(\M f)'(x_i)\), \(u_i=-f'(x_i)\) and \(v_i=f'(a(x_i))\).
Then by \(-a'(x_i)\leq D\) and \cref{eq_da} we have \(w_i\leq Dv_i\).
Applying \((\M f)''(x_i)=0\) to \cref{eq_ddmf} and inserting \cref{eq_da} we obtain \(u_iv_i=w_i^2+2w_iv_i\).
That means we can conclude the result from \cref{lem_boundbetweenzeronum}.
\end{proof}

\begin{lemma}
\label{lemma_expdecay}
Let \(D>0\) and \(u_0\leq u_1\).
Let \(b,g:[u_0,u_1]\rightarrow\mathbb{R}\) be weakly differentiable with \(b(u_1)-b(u_0)\leq D(u_1-u_0)\) such that for \(u_0< u< u_1\) we have \(g(u)>0\) and \(g'(u)\geq-b'(u)g(u)\).
Then \(g(u_1)\geq\exp(-D(u_1-u_0))g(u_0)\).
\end{lemma}

\begin{proof}
We have
\begin{align*}
&&
g'
&\geq
-b'g
\\
&\iff&
(\log\circ g)'
&\geq
-b'
\\
&\implies&
\log(g(u_1))
-
\log(g(u_0))
&\geq
b(u_0)
-
b(u_1)
\geq
-D(u_1-u_0)
\\
&\implies&
\log(g(u_1))
&\geq
\log(\exp(-D(u_1-u_0))g(u_0))
\\
&\iff&
g(u_1)
&\geq
\exp(-D(u_1-u_0))
g(u_0)
.
\end{align*}
\end{proof}

\begin{corollary}
\label{corollary_logexpdecay}
Let \(K,D>0\), \(0<x_0\leq x_1\) and let \(m,h:[x_0,x_1]\rightarrow[0,\infty)\) be nonnegative weakly differentiable functions where \(h\) is nondecreasing with \(h(x_1)-h(x_0)\leq D(x_1-x_0)\) and such that for \(x_0<x<x_1\) we have
\[
m'(x)
\geq
-\f{
K+h'(x)
}{
x+h(x)
}
m(x)
.
\]
Then
\[
m(x_1)
\geq
\Bigl(\f{x_0}{x_1}\Bigr)^{\max\{1+D/K,K+D\}}
m(x_0)
.
\]
\end{corollary}

\begin{proof}
Define \(g=m\circ\exp\) and \(b=\max\{1,K\}\log\circ(K\exp+h\circ\exp)\).
Using \(h(x_1)-h(x_0)\leq D(x_1-x_0)\) and \(1-1/t\leq\log t\leq t-1\) we obtain
\begin{align*}
	\f{
	b(\log(x_1))
	-
	b(\log(x_0))
	}{
	\max\{1,K\}
	}
	&=
	\log\Bigl(\f{
			Kx_1+h(x_1)
		}{
			Kx_0+h(x_0)
		}
	\Bigr)
	\\
	&\leq
	\log\Bigl(1+\f{
			(K+D)(x_1-x_0)
		}{
			Kx_0+h(x_0)
		}
	\Bigr)
	\\
	&\leq
	\log\Bigl(1+\f{
			(K+D)(x_1-x_0)
		}{
			Kx_0
		}
	\Bigr)
	\\
	&=
	\log\Bigl(
		\f{x_1}{x_0}
		\Bigl[1+
			\f DK\Bigl(1-\f{x_0}{x_1}\Bigr)
		\Bigr]
	\Bigr)
	\\
	&=
	\log\Bigl(\f{x_1}{x_0}\Bigr)
	+
	\log\Bigl(1+\f DK\Bigl(1-\f{x_0}{x_1}\Bigr)\Bigr)
	\\
	&\leq
	\log\Bigl(\f{x_1}{x_0}\Bigr)
	+
	\f DK\Bigl(1-\f{x_0}{x_1}\Bigr)
	\\
	&\leq
	\log\Bigl(\f{x_1}{x_0}\Bigr)
	+
	\f DK\log\Bigl(\f{x_1}{x_0}\Bigr)
	\\
	&=
	(1+D/K)(\log(x_1)-\log(x_0))
	.
\end{align*}
Moreover
\[
b'
=
\max\{1,K\}
\f{
(K+h'\circ\exp)\exp
}{
K\exp+h\circ\exp
}
\geq
\f{
(K+h'\circ\exp)\exp
}{
\exp+h\circ\exp
}
\]
and thus
\[
g'
=
(m'\circ\exp)\cdot\exp
\geq
-\f{
K+h'\circ\exp
}{
\exp+h\circ\exp
}
(m\circ\exp)
\cdot
\exp
\geq
-b'g
\]
and we can conclude from \cref{lemma_expdecay} that
\begin{align*}
m(x_1)
&=
g(\log(x_1))
\\
&\geq
\exp\bigl(-\max\{1+D/K,K+D\}[\log(x_1)-\log(x_0)]\bigr)
g(\log(x_0))
\\
&=
\Bigl(\f{x_0}{x_1}\Bigr)^{\max\{1+D/K,D+K\}}
m(x_0)
.
\end{align*}
\end{proof}

\begin{corollary}
\label{corollary_mdecayorig}
For any \(x_0\leq x_1\) we have
\[
	-(\M f)'(x_1)
	\geq
	e^{-c}
	\Bigl(
	\f{x_0}{x_1}
	\Bigr)^{2+c}
	(
	-
	(\M f)'(x)
	)
	.
\]
\end{corollary}

\begin{proof}
By \cref{eq_abound} for every \(x_0\leq x_1\) we have
\[
(-a(x_1))-(-a(x_0))
\leq
(-a(x_1))
\leq cx_1
	= \f{cx_1}{x_1-x_0}(x_1-x_0)
.
\]
	By \cref{eq_ddmf} and \(f'\leq0\) we may apply \cref{corollary_logexpdecay} with \(m=-(\M f)'\), \(h=-a\), \(K=2\) and \(D=cx_1/(x_1-x_0)\) and obtain
\[
	-(\M f)'(x_1)
	\geq
	\Bigl(
	\f{x_0}{x_1}
	\Bigr)^{2+\f{cx_1}{x_1-x_0}}
	(-(\M f)'(x_0))
	.
\]
We estimate
\[
	\Bigl(\f{x_0}{x_1}\Bigr)^{\f{x_1}{x_1-x_0}}
	=
	\f{x_0}{x_1}
	\Bigl(1+\f1{\f{x_0}{x_1-x_0}}\Bigr)^{-\f{x_0}{x_1-x_0}}
	\geq
	\f{x_0}{x_1}
	\f1e
\]
to finish the proof.
\end{proof}

\begin{corollary}
\label{corollary_mdecay}
For any \(\lambda>0\) there exists a \(d_{c,\lambda}>0\) such that for any \(x>0\) we have
\[
\sup_{
[x,e^\lambda x]
}
f'\circ a
\geq
d_{c,\lambda}
	(-(\M f)'(x))
.
\]
\end{corollary}

\begin{proof}
By \cref{eq_abound} there must be a \(x\leq y\leq e^\lambda x\) for which \(-a'(y)\leq c/(1-e^{-\lambda})\) for otherwise \(-a(e^\lambda x)>ce^\lambda x\), contradicting \cref{eq_abound}.
By \cref{eq_da,corollary_mdecayorig} we can conclude
\[
f'(a(y))
=
\f{
	-(\M f)'(y)
}{
-a'(y)
}
\geq
\f{
	e^{-c-\lambda(2+c)}(-(\M f)'(x))
}{
	c/(1-e^{-\lambda})
}
.
\]
\end{proof}

\begin{proposition}
\label{lemma_decreasinginbiginterval}
	Let \(\lambda>0\) and let \(0<x_0<e^\lambda x_0<x_1<\infty\) such that for all \(x_0\leq x\leq x_1\) we have \((\M f)''(x)\geq0\) and for \(i=0,1\) we have \((\M f)''(x_i)=0\).
Then
\[
|
	(\M f)'(x_1)
-
	(\M f)'(x_0)
|
\lesssim_{c,\lambda}
\var_{[x_0,x_1]}(f')
+
\var_{[x_0,x_1]}(f'\circ a)
.
\]
\end{proposition}

\begin{proof}
Note, that by assumption for any \(x_0\leq x\leq x_1\) we have \((\M f)'(x_0)\leq(\M f)'(x)\leq(\M f)'(x_1)\leq0\).
Thus it suffices to consider the case
\[
\var_{[x_0,x_1]}(f'\circ a)
\leq
\f{d_{c,\lambda}}2
(-(\M f)'(x_0))
\]
where \(d_{c,\lambda}\) is the constant on the right hand side in \cref{corollary_mdecay}.
As a consequence for every \(x\in[x_0,x_1]\) we have
\begin{align*}
f'(a(x))
&\geq
\sup_{[x_0,x_1]}
f'\circ a
-
\f{d_{c,\lambda}}2
(-(\M f)'(x_0))
\\
&\geq
\f{d_{c,\lambda}}2
(-(\M f)'(x_0))
\geq
\f{d_{c,\lambda}}2
(-(\M f)'(x))
.
\end{align*}
By \cref{eq_da} this means \(-a'(x)\leq2/d_{c,\lambda}\) and the result follows from \cref{proposition_lowspeed}.
\end{proof}

\begin{proposition}
\label{lemma_anyinsmallintervals}
Let \(\lambda,\mu>0\) and let \(0< x_0\leq x_1\leq e^\lambda x_0<\infty\) and for \(i=0,1\) assume \((\M f)''(x_i)=0\).
Then
\[
	\int_{\{x\in[x_0,x_1]:(\M f)''(x)>0\}}
	(\M f)''
\lesssim_{c,\lambda,\mu}
\var_{[x_0,x_1]}(f')
+
\var_{[x_0,e^\mu x_1]}(f'\circ a)
.
\]
\end{proposition}

\begin{proof}
	Denote \(M=\sup_{[x_0,x_1]}(-(\M f)')\).
Then by \cref{eq_ddmf} for any \(x_0\leq x\leq x_1\) we have
\[
	(\M f)''(x)
\leq
	\f{-(2-a'(x))(\M f)'(x)}{x-a(x)}
\leq
\f{2-a'(x)}{x_0}M
.
\]
This implies
\begin{align*}
	\int_{\{x\in[x_0,x_1]:(\M f)''(x)>0\}}
	(\M f)''
	&\leq
	\int_{x_0}^{e^\lambda x_0}
	\f{2-a'}{x_0}M
	\\
	&\leq
	\Bigl[
		\f{
			2(e^\lambda x_0-x_0)
		}{x_0}
		+
		\f{
			a(x_0)-a(e^\lambda x_0)
		}{x_0}
		\Bigr]
	M
	\\
	&\leq
	[2(e^\lambda-1)+ce^\lambda]
	M
	.
\end{align*}
That means it suffices to consider the case
\[
\var_{[x_0,e^\mu x_1]}(f'\circ a)
\leq
\f{d_{c,\mu}}2
M
,
\]
where \(d_{c,\mu}\) is the constant on the right hand side in \cref{corollary_mdecay}.
As a consequence for every \(x\in[x_0,x_1]\) we have
\begin{align*}
f'(a(x))
&\geq
\sup_{[x_0,e^\mu x_1]}
f'\circ a
-
\f{d_{c,\mu}}2
M
=
\sup_{x_0\leq y\leq x_1}
\sup_{[y,e^\mu y]}
f'\circ a
-
\f{d_{c,\mu}}2
M
\\
&\geq
\f{d_{c,\mu}}2
M
\geq
\f{d_{c,\mu}}2
(-(\M f)'(x))
\end{align*}
which by \cref{eq_da} means \(-a'(x)\leq2/d_{c,\mu}\).
Since \((\M f)''\) is continuous we can write
\(
\{x\in(x_0,x_1):(\M f)''(x)>0\}
=
\bigcup_{k=1}^\infty
(y_k,z_k)
\)
and by the fundamental theorem of calculus and \cref{proposition_lowspeed} we get
\begin{align*}
\int_{\{x\in[x_0,x_1]:(\M f)''(x)>0\}}
(\M f)''
&=
\sum_{k=1}^\infty
(\M f)'(z_k)-(\M f)'(y_k)
\\
&\lesssim_{c,\mu}
\sum_{k=1}^\infty
|f'(z_k)-f'(y_k)|
+
|f'(a(z_k))-f'(a(y_k))|
\\
&\leq
\var_{[x_0,x_1]}(f')
+
\var_{[x_0,x_1]}(f'\circ a)
.
\end{align*}
\end{proof}

\begin{lemma}
\label{lemma_sequence}
For any compact set \(Z\subset\mathbb{R}\) there exist an \(N\in\mathbb{N}\) and \(u_1,\ldots,u_N\in Z\) with \(u_1=\inf Z,\ u_N=\sup Z\), \(u_k<u_{k+1}\), \(u_k+1< u_{k+2}\) and such that
for each \(k=1,\ldots,N-1\) we have \(u_k+3>u_{k+1}\) or \((u_k,u_{k+1})\cap Z=\emptyset\).
\end{lemma}

\begin{proof}
The set \(U=[\inf Z,\sup Z]\setminus Z\) is open and bounded.
Enumerate by \((t_{2i-1},t_{2i})\subset U\) the maximal open subintervals of \(U\) of length at least \(2\).
We order them along the real line, denote by \(M\) the number of such intervals and set \(t_0=\inf Z\) and \(t_{2M+1}=\sup Z\).
Then \(t_0,\ldots,t_{2M+1}\in Z\) and for each \(i\) we have \(t_{2(i-1)}\leq t_{2i-1}\leq t_{2i}-2\) and \(U\setminus\bigcup_{i=1}^M(t_{2i-1},t_{2i})\) contains no interval of length \(2\).
For each \(i=0,\ldots,M\) with \(t_{2i+1}>t_{2i}\) there is an \(N_i\geq0\) and \(t_{2i}^0,\ldots,t_{2i}^{N_i}\in Z\) with \(t_{2i}^0=t_{2i}\), \(t_{2i}^{j+1}-t_{2i}^j\in(1,3)\) and \(t_{2i+1}-t_{2i}^{N_i}\in(0,3)\).
Then \(\{t_{2i}^j,t_{2i+1}:i=0,\ldots,M,\ j=0,\ldots,N_i\}\) is finite, and we enumerate it increasingly as \(u_1<\ldots<u_N\), excluding duplicates.

For \(k\) with \(u_k=t_{2i+1}\) for some \(i\) we have \(u_{k+1}=t_{2(i+1)}\) and thus the interval \((u_k,u_{k+1})\) does not intersect \(Z\) and has length at least \(2\).
For \(k\) with \(u_k=t_{2i}^j\) for some \(i\) and \(j<N_i\) we have \(u_{k+1}=t_{2i}^{j+1}\in(u_k+1,u_k+3)\).
For \(j=N_i\) we have \(u_{k+1}=t_{2i+1}<u_k+3\) and \(u_{k+2}=t_{2(i+1)}\geq u_{k+1}+2\).
\end{proof}

\begin{theorem}
\label{theorem_smooth}
For every \(c>0\) there exists a \(C>0\) such that for every \(K>0\) and every \(f:\mathbb{R}\rightarrow\mathbb{R}\) that satisfies \cref{eq_diffboundassumption}, \(|a(x)|\leq c|x|\) and is twice differentiable on \(\mathbb{R}\setminus\{0\}\) we have
\[
\var((\M f)')
\leq C
\var(f')
.
\]
\end{theorem}

\begin{remark}
\Cref{theorem_smooth} is \cref{theorem_general} under the additional assumptions that \(f\) satisfies \cref{eq_diffboundassumption} and is twice differentiable on \(\mathbb{R}\setminus\{0\}\).
\end{remark}

\begin{proof}[Proof of \Cref{theorem_smooth}]
\Cref{eq_diffboundassumption} implies that \(f\) is continuous and that it is increasing on \((-\infty,0)\) and decreasing on \((0,\infty)\).
As a consequence the same is true for \(\M f\) and \(\M f(0)=f(0)\).
\Cref{lem_mandadifferentiable} implies that \(\M f\) is differentiable
everywhere except in \(0\).
By \cref{eq_diffmf} for \(x>0\) we can infer
\[
	|(\M f)'(x)|
	=
	-(\M f)'(x)
	=
	\f{\M f(x)-f(x)}{x-a(x)}
	\leq
	\f{f(0)-f(x)}x
	\leq
	\sup_{0<y<x}
	(-f'(y))
\]
and similarly for \(x<0\) we have
\(
	|(\M f)'(x)|
\leq
\sup_{x<y<0}
f'(y)
.
\)
For every \(x\neq0\) we can conclude
\begin{equation}
\label{eq_dmflinftybound}
	|(\M f)'(x)|
\leq
\var(f')
.
\end{equation}
Moreover we infer that \(\M f\) has a global weak derivative which is pointwise defined everywhere except in \(0\).
Thus \(\var((\M f)')\) is well defined with
\begin{align}
\nonumber
\var((\M f)')
&=
\lim_{n\rightarrow\infty}
\var_{(-n,n)}((\M f)')
\\
\nonumber
&=
\lim_{n\rightarrow\infty}
\var_{(-n,-1/n)}((\M f)')
+
|(\M f)'(-1/n)-(\M f)'(1/n)|
\\
\label{eq_globalversplit}
&\qquad\qquad+
\var_{(1/n,n)}((\M f)')
.
\end{align}
\Cref{lem_mandadifferentiable} implies that \(\M f\) is even three times differentiable on \(\mathbb{R}\setminus\{0\}\) and thus the possibly infinite Radon measure \((\M f)''\) is represented by a continuous function away from \(0\).
In particular, the set
\[
Z_n
=
\{1/n\leq x\leq n:(\M f)''(x)=0\}
\]
is a well defined compact set.
If \(Z_n\neq\emptyset\) then abbreviate
\(
A_n
=
\inf Z_n
\)
and
\(
B_n
=
\sup Z_n
,
\)
and otherwise set \(A_n=B_n=1\) so that \((A_n,B_n)=\emptyset\).
Then
\begin{align}
\nonumber
\var_{(1/n,n)}((\M f)')
&=
|(\M f)'(1/n)-(\M f)'(A_n)|
+
\var_{(A_n,B_n)}((\M f)')
\\
\label{eq_restrictvartozeros}
&\qquad+
|(\M f)'(B_n)-(\M f)'(n)|
,
\end{align}
and since
\begin{align*}
	&
	(\M f)'(B_n)
	-
	(\M f)'(A_n)
	=
	\int_{A_n}^{B_n}
	(\M f)''
	\\
	&\qquad\qquad=
	\int_{A_n<x<B_n:(\M f)''(x)>0}
	|(\M f)''|
	-
	\int_{A_n<x<B_n:(\M f)''(x)\leq0}
	|(\M f)''|
\end{align*}
we have
\begin{align}
	\nonumber
	&
	\var_{(A_n,B_n)}((\M f)')
	\\
	\nonumber
	&\qquad=
	\int_{A_n<x<B_n:(\M f)''(x)\leq0}
	|(\M f)''|
	+
	\int_{A_n<x<B_n:(\M f)''(x)>0}
	|(\M f)''|
	\\
	\label{eq_reducetopositive}
	&\qquad=
	2\int_{A_n<x<B_n:(\M f)''(x)>0}
	|(\M f)''|
	+
	(\M f)'(A_n)
	-
	(\M f)'(B_n)
	,
\end{align}
By symmetry and \cref{eq_globalversplit,eq_restrictvartozeros,eq_reducetopositive,eq_dmflinftybound} it is enough to bound
\[
	\int_{A_n<x<B_n:(\M f)''(x)>0}
	(\M f)''
	\lesssim_c
	\var(f')
\]
in order to finish the proof.

We apply \cref{lemma_sequence} to the compact set
\(
\log(Z_n)
\)
and with \(x_k=\exp(u_k)\) we obtain \(x_1<\ldots<x_N\) with \((\M f)''(x_k)=0\), \(x_1=A_n\), \(x_N=B_n\) and \(x_{k+2}\geq ex_k\) such that for \(k=1,\ldots,N-1\) we have
\begin{enumerate}
\item
\label{item_smallintervals}
\(x_{k+1}\leq e^3 x_k\) or
\item
\label{item_bigintervals}
\((\M f)''\) has no zero between \(x_k\) and \(x_{k+1}\).
\end{enumerate}
By \cref{lemma_anyinsmallintervals} in case \cref{item_smallintervals} and by \cref{lemma_decreasinginbiginterval} in case \cref{item_bigintervals} we obtain
\begin{equation}
\label{eq_intervalbound}
\int_{\{x\in[x_k,x_{k+1}]:(\M f)''(x)>0\}}
(\M f)''
\lesssim_c
\var_{[x_k,x_{k+1}]}(f')
+
\var_{[x_k,ex_{k+1}]}(f'\circ a)
.
\end{equation}
Note, that \(a\) is monotone on \((0,\infty)\), and since \(u_{k+2}\geq ex_k\) for every \(x>0\) there can be at most two different indices \(k\) with \(x_k<x<x_{k+1}\).
So by summing \cref{eq_intervalbound} over \(k=1,\ldots,N-1\) we finish the proof,
\begin{align*}
	\int_{\{A_n\leq x\leq B_n:(\M f)''(x)>0\}}
	(\M f)''
&\lesssim_c
\sum_{k\in\mathbb{Z}}
\bigl(
\var_{[x_k,x_{k+1}]}(f')
+
\var_{[x_k,ex_{k+1}]}(f'\circ a)
\bigr)
\\
&\leq
\var_{(0,\infty)}(f')
+
2
\var_{(0,\infty)}(f'\circ a)
\leq
2\var f
.
\end{align*}
\end{proof}

\section{Approximation}
\label{sec_approximation}

\begin{lemma}
	\label{lemma_approximation}
	Let \(f:\mathbb{R}\rightarrow\mathbb{R}\) be nondecreasing on \((-\infty,0)\), nonincreasing on \((0,\infty)\) and weakly differentiable with \(\var(f')<\infty\).
	Then there exists a sequence of functions \(f_1,f_2,\ldots\) of the same class which moreover are smooth on \(\mathbb{R}\setminus\{0\}\), converge to \(f\) locally uniformly, satisfy
	\[
		\limsup_{n\rightarrow\infty}
		\var(f_n')
		\leq
		\var(f')
		,
	\]
	\begin{equation}
		\label{eq_anbound}
		\min\Bigl\{
		\sup_{n\in\mathbb{N}}
		\sup_{x\neq0}
		|a_n(x)|/|a(x)|
		,
		\limsup_{n\rightarrow\infty}
		\sup_{x\neq0}
		2|a_n(x)|/|x|
		\Bigr\}
		\leq
		1
		,
	\end{equation}
	and for all \(n\in\mathbb{N}\) and \(x\neq0\)
	\begin{equation}
		\label{eq_fndiffbound}
		1/n
		\leq
		-\sign(x)f_n'(x)
		\leq
		n
		.
	\end{equation}
\end{lemma}

\begin{proof}
~\paragraph{Definitions.}
Set \(D=\esssup|f'|\).
Since \(\var(f')<\infty\) we have \(D<\infty\).
Take \(\varphi:\mathbb{R}\rightarrow[0,\infty)\) smooth, supported on \([-1,1]\), even, decreasing away from the origin and with \(\int\varphi=1\), denote \(\alpha_k=2\sup_x|\varphi^{(k)}(x)|\) and set
\[
\psi(x)
=
2
\int_0^{|x|}
\varphi
.
\]
For \(t>0\) denote \(\varphi_t(x)=t^{-1}\varphi(x/t)\) and \(\psi_t(x)=t^2\psi(x/t)\) and for \(n\in\mathbb{N}\) define
\begin{align}
\nonumber
g_n(x)
&=
(f*\varphi_{\psi_{1/n}(x)})(x)
\\
\nonumber
&=
\int
\varphi_{\psi_{1/n}(x)}(y) f(x+y)
\intd y
\\
\label{eq_rescaledconvolution}
&=
\int_{-1}^1
\varphi(y) f(x+\psi_{1/n}(x)y)
\intd y
\end{align}
and
\[
f_n(x)
=
g_n(x)
-
|x|
/
\sqrt n
.
\]

\paragraph{Convergence.}
For every \(x\neq0\) we have
\begin{align}
	\label{eq_psiconst}
	\psi_{1/n}(x)
	&=
	1/n^2
	&x&\geq1/n
	,\\
	\label{eq_psidiff}
	\sign(x)
	\psi_{1/n}'(x)
	&=
	\varphi(n|x|)/n
	\in
	[0,\alpha_0/n]
	,
	\\
	\label{eq_psidiffdiff}
	|
	\psi_{1/n}''(x)
	|
	&=
	|
	\varphi'(n|x|)
	|
	\leq
	\alpha_1
	.
\end{align}
For any \(-1\leq y\leq 1\) we have \(|f(x+\psi_{1/n}(x)y)-f(x)|\leq D\psi_{1/n}(x)\) and thus by \cref{eq_psidiff,eq_rescaledconvolution} we have
\begin{equation}
\label{eq_gnminusf}
|g_n(x)-f(x)|
\leq
D\psi_{1/n}(x)
\leq
\alpha_0D|x|/n
.
\end{equation}
Since \(g_n(x)-f_n(x)=|x|/\sqrt n\) we can conclude \(f_n\rightarrow f\) locally uniformly.

\paragraph{Smoothness.}
The map \(x\mapsto\psi_{1/n}(x)\) is monotone and smooth on \((0,\infty)\) with \(|\psi_{1/n}^{(k)}(x)|\leq n^{k-2}\alpha_k\), and for any \(0<x_0<x_1<\infty\) we have \(t_i\seq\psi_{1/n}(x_i)>0\) and \(\psi_{1/n}([x_0,x_1])=[t_0,t_1]\).
The support of \((t,z)\mapsto\varphi_t(z)\) restricted to \([t_0,t_1]\times\mathbb{R}\) belongs to the compact set \([t_0,t_1]\times[-t_1,t_1]\), on which \((t,z)\mapsto\varphi_t(z)\) has uniform derivative bounds since it is smooth on \((0,\infty)\times\mathbb{R}\).
We can conclude that \((x,y)\mapsto\varphi_{\psi_{1/n}(x)}(y-x)\) is smooth with derivatives bounded uniformly on \([x_0,x_1]\times\mathbb{R}\) and support restricted to \([x_0,x_1]\times\mathbb{R}\) compact.
Since \(f\) is locally integrable this implies that \(g_n,f_n\) are smooth on \(\mathbb{R}\setminus\{0\}\).

\paragraph{Variation bound.}
In particular by \cref{eq_psiconst} for any \(x>1/n\) we have
\(
g_n''(x)
=
(f''*\varphi_{1/n^2})(x)
,
\)
where \(f''\) is a Radon measure, see \cref{subsec_definitions}, and \(|f''|\) its total variation measure.
This implies
\begin{equation}
	\label{eq_awayfromzero}
	\int_{1/n}^\infty|g_n''|
	\leq
	\int_{1/n}^\infty|f''|*\varphi_{1/n^2}
	\leq
	\int_{1/n-1/n^2}^\infty|f''|
	.
\end{equation}
For \(-1\leq y\leq 1\) denote \(u_{n,y}:\mathbb{R}\rightarrow\mathbb{R},\ u_{n,y}(x)=x+\psi_{1/n}(x)y\) so that by \cref{eq_rescaledconvolution} for any \(x\neq0\) we have
\begin{equation}
\label{eq_gndiff}
g_n'(x)
=
\int_{-1}^1
\varphi(y) f'(u_{n,y}(x))u_{n,y}'(x)
\intd y
.
\end{equation}
By \cref{eq_psidiff} we have
\begin{equation}
\label{eq_coordtransformdiff}
u_{n,y}'(x)
=
1+y\psi_{1/n}'(x)
\in
[1-\alpha_0/n,1+\alpha_0/n]
,
\end{equation}
so that since \(u_{n,y}(0)=0\) for \(n>\alpha_0\) we have \(\sign(u_{n,y}(x))=\sign(x)\), and by \cref{eq_gndiff} we obtain
\begin{align*}
	0
	&\leq
	-\sign(x)g_n'(x)
	\leq
	(1+\alpha_0/n)D
	,\\
	1/\sqrt n
	&\leq
	-\sign(x)f_n'(x)
	\leq
	(
	1
	+
	\alpha_0/n
	)
	D
	+
	1/\sqrt n
	,
\end{align*}
which ensures \cref{eq_fndiffbound} for a subsequence.
By \cref{eq_gndiff} we have
\[
g_n''(x)
=
\int_{-1}^1
\varphi(y)
\Bigl[
f''(u_{n,y}(x))(u_{n,y}'(x))^2
+
f'(u_{n,y}(x))u_{n,y}''(x)
\Bigr]
\intd y
,
\]
and thus by \cref{eq_psidiffdiff,eq_coordtransformdiff} we obtain
\begin{align*}
	\int_0^{1/n}
	|g_n''|
	&\leq
	\int_0^{1/n}
	\int_{-1}^1
	\varphi(y)
	\Bigl[
		\bigl|
		f''(u_{n,y}(x))
		\bigr|
		(1+\alpha_0/n)^2
		+
		\alpha_1D
	\Bigr]
	\intd y
	\intd x
	\\
	&\leq
	\f{\alpha_1D}n
	+
	(1+\alpha_0/n)^2
	\int_{-1}^1
	\varphi(y)
	\int_0^{1/n}
	\bigl|
	f''(u_{n,y}(x))
	\bigr|
	\intd x
	\intd y
	,
\end{align*}
where by \cref{eq_coordtransformdiff} we have
\begin{align*}
\int_0^{1/n}
|f''(u_{n,y}(x))|
\intd x
&=
\int_{u_{n,y}(0)}^{u_{n,y}(1/n)}
\f{
|f''(z)|
}{
u_{n,y}'(u_{n,y}^{-1}(z))
}
\intd z
\\
&\leq
\f1{1-\alpha_0/n}
\int_0^{1/n+1/n^2}
|f''(z)|
\intd z
.
\end{align*}
That means
\begin{equation}
	\label{eq_nearzero}
	\int_0^{1/n}
	|g_n''|
	\leq
	\f{\alpha_1D}n
	+
	\f{(1+\alpha_0/n)^2}{1-\alpha_0/n}
	\int_0^{1/n+1/n^2}
	|f''|
	.
\end{equation}
From \cref{eq_awayfromzero,eq_nearzero} and since \(f_n''=g_n''\) on \(\mathbb{R}\setminus\{0\}\) and \(\var(f')<\infty\) we can conclude
\begin{align*}
	\limsup_{n\rightarrow\infty}
	\var_{(0,\infty)}(f_n')
	&=
	\limsup_{n\rightarrow\infty}
	\int_0^\infty|f_n''|
	=
	\limsup_{n\rightarrow\infty}
	\int_0^{1/n}|g_n''|
	+
	\int_{1/n}^\infty|g_n''|
	\\
	&\leq
	\int_0^\infty|f''|
	+
	\limsup_{n\rightarrow\infty}
	\int_{1/n-1/n^2}^{1/n+1/n^2}
	|f''|
	\\
	&=
	\var_{(0,\infty)}(f')
\end{align*}
and by symmetry
\[
	\limsup_{n\rightarrow\infty}
	\var_{\mathbb{R}\setminus\{0\}}(f_n')
	\leq
	\var_{\mathbb{R}\setminus\{0\}}(f')
	.
\]
By \(\var_{\mathbb{R}\setminus\{0\}}(f')<\infty\), \(f'(x)\) converges to some real values \((f')_-(0)\) and \((f')_+(0)\) for \(x\uparrow0\) and \(x\downarrow0\).
The same is true for \(f_n\) and thus
\[
\var(f_n')
=
|
(f_n')_+(0)
-
(f_n')_-(0)
|
+
\var_{\mathbb{R}\setminus\{0\}}(f_n')
.
\]
Since by \cref{eq_psiconst,eq_psidiff,eq_coordtransformdiff} we have \(u_{n,y}(x)\rightarrow x,\ u_{n,y}'(x)\rightarrow1\) for \(n\rightarrow\infty\) uniformly in \(x\in\mathbb{R},\ y\in[-1,1]\), it follows from \cref{eq_gndiff} that \((f_n')_\pm(0)\rightarrow(f')_\pm(0)\) as \(n\rightarrow\infty\) and we can conclude
\[
	\limsup_{n\rightarrow\infty}
	\var(f_n')
	\leq
	|
	f_+(0)
	-
	f_-(0)
	|
	+
	\var_{\mathbb{R}\setminus\{0\}}(f')
	=
	\var(f')
	.
\]

\paragraph{Bound on \(a_n\).}
For any \(y<0<x\) by \cref{eq_gnminusf} we have
\begin{align*}
\int_y^x f
-
\int_y^x f_n
&=
\int_y^x f-g_n
+
\int_y^x g_n-f_n
\\
&\leq
\int_y^x
\f{
D\alpha_0|z|
}n
\intd z
+
\int_y^x
\f{|z|}{\sqrt n}
\intd z
=
\f{\sqrt n+D\alpha_0}{2n}(x^2+y^2)
\end{align*}
and inserting \(y=a(x)\) we obtain
\begin{equation}
\label{eq_mfngeqmf}
\M f_n(x)
\geq
\f1{x-a(x)}
\int_{a(x)}^x f_n
\geq
\M f(x)
-
\f{\sqrt n+D\alpha_0}{2n}
\f{
x^2+a(x)^2
}{
x-a(x)
}
.
\end{equation}
Denote
\[
	A(x)
	=
	\max\Bigl\{
	\f{\sqrt n+D\alpha_0}{2(\sqrt n-D\alpha_0)}
	\f{
	x^2+a(x)^2
	}{
	x-a(x)
	}
	,
	-a(x)
	\Bigr\}
	.
\]
Then for \(-y\geq A(x)\) by \cref{eq_a} and since \(f\) is nondecreasing on \((-\infty,0)\) we have \(f(y)\leq\M f(x)\) and together with \cref{eq_gnminusf,eq_mfngeqmf} we obtain
\begin{align*}
f_n(y)
&=
g_n(y)
-
(-y)/\sqrt n
\\
&\leq
f(y)
-
(\sqrt n-D\alpha_0)(-y)/n
\\
&\leq
\M f(x)
-
(\sqrt n-D\alpha_0)(-y)/n
\\
&\leq
\M f_n(x)
,
\end{align*}
which means \(-a_n(x)\leq A(x)\).
Since \(\lim_{n\rightarrow\infty}(\sqrt n+D\alpha_0)/(\sqrt n-D\alpha_0)=1\) and \((x^2+a(x)^2)/(x-a(x))\leq\max\{x,-a(x)\}\) we can conclude \cref{eq_anbound} by symmetry.
\end{proof}

\begin{lemma}
\label{lemma_Mfconverge}
Let \(f_n\rightarrow f\) locally uniformly with \(|a(x)|,|a_n(x)|\leq cx\).
Then \(\M f_n\rightarrow\M f\) locally uniformly.
\end{lemma}

\begin{proof}
Let \(|x|\leq r\) and \(\varepsilon>0\).
Take \(N\in\mathbb{N}\) such that for all \(n\geq N\) and \(y\in[-cr,r]\) we have \(|f_n(y)-f(y)|\leq \varepsilon\).
Then for any \(n\geq N\) we have
\[
\M f(x)
=
\sup_{-cr\leq y\leq 0}
\f1{|x-y|}
\int_y^x
f
\geq
\sup_{-cr\leq y\leq 0}
\f1{|x-y|}
\int_y^x
f_n
-
\varepsilon
\geq
\M f_n(x)
-
\varepsilon
.
\]
We can show \(\M f_n(x)\geq\M f(x)-\varepsilon\) the same way.
\end{proof}

\begin{lemma}
\label{lemma_varapprox}
Assume \(f_n\rightarrow f\) in \(L^1_\loc\).
Then
\[
\var(f')
\leq
\liminf_{n\in\mathbb{N}}
\var(f_n')
.
\]
\end{lemma}

\begin{proof}
Let \(\varphi\) with \(\|\varphi\|_\infty\leq1\) be smooth and compactly supported.
Then
\[
-\int f'\varphi'
=
\int f\varphi''
=
\lim_{n\rightarrow\infty}
\int f_n\varphi''
=
-
\lim_{n\rightarrow\infty}
\int f_n'\varphi'
\leq
\liminf_{n\rightarrow\infty}
\var(f_n')
.
\]
We take the supremum over all such \(\varphi\) to obtain \(\var(f')\) and conclude the proof.
\end{proof}

\begin{proof}[Proof of \cref{theorem_general}]
Let \(f_1,f_2,\ldots\) be the sequence of functions from \cref{lemma_approximation}.
They satisfy the assumptions in \cref{theorem_smooth}, albeit with constant \(\max\{c,1/2+\varepsilon\}\) instead of \(c\).
Moreover, \(f_n\) converges to \(f\) locally uniformly which by \cref{lemma_Mfconverge} implies that \(\M f_n\) converges to \(\M f\) locally uniformly.
Thus using \cref{lemma_varapprox} we can conclude
\[
\var((\M f)')
\leq
\liminf_{n\rightarrow\infty}
\var((\M f_n)')
\lesssim_c
\liminf_{n\rightarrow\infty}
\var(f_n')
\leq
\var(f')
.
\]
\end{proof}

\section{Counterexample}
\label{sec_counterexample}

Let \(\varepsilon>0\), \(N\in\mathbb{N}\) be odd and let \(M\in\mathbb{N}\).
Let \(f:\mathbb{R}\rightarrow(-\infty,0]\) be the unique continuous function with \(f(0)=0\) and
\[
f'(x)=
\begin{cases}
-1 & x>0\\
\f\varepsilon M & -1<x<0\\
\varepsilon & -M^{n+1}<x<-M^n,\ 0\leq n\leq N\tx{ even}\\
\f\varepsilon{M^{n+1}} & -M^{n+1}<x<-M^n,\ 1\leq n\leq N\tx{ odd}\\
\f\varepsilon{M^{N+1}} & x<-M^{N+1}
,
\end{cases}
\]
and let
\[
g(x)=
\begin{cases}
-x & x>0\\
\f{x\varepsilon}M & -1<x\leq0\\
(M^n-M^{n-1}+x)\varepsilon & -M^{n+1}<x\leq-M^n,\ 0\leq n\leq N\tx{ even}\\
-M^n\varepsilon & -M^{n+1}<x\leq-M^n,\ 1\leq n\leq N\tx{ odd}\\
-M^N\varepsilon & x\leq-M^{N+1}
\end{cases}
\]

For \(r>0\) denote \(P_r=\{(\varepsilon,N,M):\varepsilon^2+1/N^2+1/M^2<r^2\}\).
In this \lcnamecref{sec_counterexample} any instance of \(\ord(\varphi)\) for some formula \(\varphi\) in \(\varepsilon,N,M\) stands for a different function \(F:P_r\rightarrow\mathbb{R}\) for some \(r>0\) such that
\[
\sup_{P_r}
\f{
|F(\varepsilon,N,M)|
}{
|\varphi(\varepsilon,N,M)|
}
<
\infty
.
\]
For \(t\neq0\) and formulas \(\varphi,\psi\) it satisfies the linearity rules \(\ord(t\cdot\varphi)=\ord(\varphi)\) and \(\ord(\varphi)\Box\ord(\psi)=\ord(\varphi\Box \psi)\) for \(\Box\in\{+,*\}\) in the sense that for the functions on the left hand side appropriate functions on the right hand side exist and vice versa.
Let \(\varphi\) with \(\lim_{r\rightarrow0}\inf_{P_r}\varphi(\varepsilon,N,M)>-1\) and \(\lim_{r\rightarrow0}\sup_{P_r}\varphi(\varepsilon,N,M)<\infty\) and let \(t>0\).
Since the map \(x\mapsto(1+x)^t\) is strictly monotone and has derivative near \(t\) in a neighborhood of \(0\) we also have
\begin{equation}
\label{eq_ordpower}
(1+\ord(\varphi))^t
=
1+\ord(\varphi)
.
\end{equation}
We will also frequently use that a geometric sum is dominated by its the largest summand,
\[
\sum_{k\in\mathbb{Z},\ k\leq n}
M^k
=
\f{M^n}{1-1/M}
=
(1+\ord(1/M))
M^n
.
\]

\begin{lemma}
\label{lem_fsimg}
For any \(x>-M^{N+1}\) we have \(f(x)=(1+\ord(1/M))g(x)\).
\end{lemma}

\begin{proof}
For all \(x>-1\) we have \(f(x)=g(x)\).
For even \(0\leq n\leq N\) we have
\[
f(-M^n)
-
f(-M^{n+1})
=
\int_{-M^{n+1}}^{-M^n}
f'
=
(1-1/M)M^{n+1}\varepsilon
\]
and for odd \(1\leq n\leq N\) we have
\[
f(-M^n)
-
f(-M^{n+1})
=
(1-1/M)\varepsilon
.
\]
That means for every \(0\leq n\leq N\) we have
\begin{align}
\nonumber
f(-M^n)
&=
f(-1)
+
\sum_{k=0}^{n-1}
f(-M^{k+1})
-
f(-M^k)
\\
\nonumber
&=
-\f\varepsilon M
-\varepsilon(1-1/M)
\Bigl(
\sum_{k=0,2,\ldots,2\lceil n/2\rceil-2}
M^{k+1}
+
\sum_{k=1,3,\ldots,2\lfloor n/2\rfloor-1}
1
\Bigr)
\\
\nonumber
&=
-\varepsilon
(1+\ord(1/M))
M^{2\lceil n/2\rceil-1}
\\
\label{eq_fgatMn}
&=
(1+\ord(1/M))
g(-M^n)
.
\end{align}
Let \(0\leq n\leq N\) and \(-M^{n+1}<x<-M^n\).
If \(n\) is even then \(f'(x)=\varepsilon=g'(x)\) and we can conclude from \cref{eq_fgatMn} that
\begin{align*}
|f(x)-g(x)|
&=
|f(-M^n)-g(-M^n)|
=
\ord(1/M)
|f(-M^n)|
\leq
\ord(1/M)
|f(x)|
.
\end{align*}
If \(n\) is odd then using the definitions of \(f_n,g_n\) and \cref{eq_fgatMn} repeatedly we obtain
\begin{align*}
|f(x)-g(x)|
&\leq
|f(x)-f(-M^n)|
+
|f(-M^n)-g(-M^n)|
\\
&\leq
\varepsilon
+
\ord(1/M)
|f(-M^n)|
\leq
\ord(1/M)
(
|g(-M^n)|
+
|f(-M^n)|
)
\\
&\leq
\ord(1/M)
|f(x)|
.
\end{align*}
\end{proof}

\begin{lemma}
\label{lem_gint}
For any odd \(1\leq n\leq N\) and \(-M^{n+1}\leq x\leq-M^n\) we have
\[
\int_x^0
g
=
(1+\ord(1/M))
M^n\varepsilon
(M^n/2+x)
.
\]
\end{lemma}

\begin{proof}
Let \(1\leq k\leq n\) be odd and let \(-M^{k+1}\leq y\leq -M^k\).
Then
\begin{equation}
\label{eq_ytoMk}
\int_y^{-M^k}
g
=
M^k\varepsilon
(M^k+y)
\end{equation}
which equals \((1+\ord(1/M))M^{2k+1}\varepsilon\) if \(y=-M^{k+1}\).
For even \(0\leq k\leq n-1\) we have
\begin{align*}
\int_{-M^{k+1}}^{-M^k}
g
&=
(M^{k+1}-M^k)
(M^k-M^{k-1})\varepsilon
+
(
M^{2k}
-
M^{2(k+1)}
)
\varepsilon/2
\\
&=
-(1/2+\ord(1/M))M^{2(k+1)}\varepsilon
.
\end{align*}
Therefore
\begin{equation}
\label{eq_Mntozero}
\int_{-M^n}^0
g
=
\sum_{k=0}^{n-1}
\int_{-M^{k+1}}^{-M^k}
g
+
\int_{-1}^0
g
=
-(1/2+\ord(1/M))M^{2n}\varepsilon
\end{equation}
and the result follows from adding \cref{eq_ytoMk} with \(k=n,y=x\) to \cref{eq_Mntozero}.
\end{proof}

Note, that \cref{eq_diffboundassumption} holds for \(f\) with \(K=M^{N+1}/\varepsilon\).
That means a map \(a:(0,\infty)\rightarrow(-\infty,0)\) exists such that for all \(x>0\) we have
\[
\M f(x)
=
\f1{x-a(x)}
\int_{a(x)}^x f
\]
and which satisfies \cref{eq_a} and is Lipschitz by \cref{lemma_lipschitz}.
Moreover by \cref{lem_mandadifferentiable} for all \(x>0\) we have \cref{eq_diffmf}, and if \(-1<a(x)<0\) or \(-M^{n+1}<a(x)<-M^n\) then also \cref{eq_da,eq_ddmf} hold.

\begin{lemma}
\label{lem_exax}
For any odd \(1\leq n\leq N\) and \(x>0\) with \(-M^{n+1}\leq a(x)\leq-M^n\) we have
\[
x
=
(1+\ord(\sqrt\varepsilon+|a(x)|/M^{n+1}))M^n\sqrt\varepsilon
.
\]
\end{lemma}

\begin{proof}
Abbreviate \(a=a(x)\).
Then by \cref{lem_fsimg} we have
\[
(1+\ord(1/M))
g(a)
=
f(a)
=
\M f(x)
=
\f1{x-a}
\int_a^xf
=
\f{1+\ord(1/M)}{x-a}
\int_a^xg
.
\]
From the previous equality, the definition of \(g\) and \cref{lem_gint} we obtain
\begin{align*}
-(1+\ord(1/M))M^n\varepsilon
&=
-\f{
(1+\ord(1/M))[
M^n\varepsilon(-a-M^n/2)
+
x^2/2
]
}{
x
-
a
}
\\
x^2/2
-
(1+\ord(1/M))M^n\varepsilon x
&=
(1/2+\ord(1/M))M^{2n}\varepsilon
+
\ord(M^{n-1}\varepsilon a)
\\
(x-(1+\ord(1/M))M^n\varepsilon)^2
&=
(1+\ord(\varepsilon))(1+\ord(a/M^{n+1}))M^{2n}\varepsilon
\end{align*}
and thus by \cref{eq_ordpower} we can conclude
\begin{align*}
x
&=
(1+\ord(1/M))M^n\varepsilon
\pm
(1+\ord(\varepsilon))(1+\ord(a/M^{n+1}))M^n\sqrt\varepsilon
\\
&=
(1+\ord(\sqrt\varepsilon+|a|/M^{n+1}))M^n\sqrt\varepsilon
.
\end{align*}
\end{proof}

\begin{lemma}
\label{lem_examfp}
For any odd \(1\leq n\leq N\) and \(x>0\) with \(-M^{n+1}\leq a(x)\leq-M^n\) we have
\[
	(\M f)'(x)
=
(1+\ord(\sqrt\varepsilon+|a(x)|/M^{n+1}))
\f{
M^n\sqrt\varepsilon
}{
a(x)
}
.
\]
\end{lemma}

\begin{proof}
Insert 
\(\M f(x)=f(a(x))\) and \(f(x)=-x\) into
\cref{eq_diffmf}.
Observe that by \cref{lem_fsimg,lem_exax} and the definition of \(g\) the term \(f(a(x))\) in the nominator is negligible against \(-x\), and that \(x\) in the denominator is negligible against \(-a(x)\).
That means
\(
(\M f)'(x)
=
(1+\ord(\ldots))x/a(x)
\)
where \(\ord(\ldots)\) is dominated by the \(\ord(\ldots)\) terms in the formula for \(x\) in \cref{lem_exax}.
\end{proof}

\begin{corollary}
\label{corollary_dmfvalues}
For any \(0\leq n\leq N\) and \(x>0\) with \(a(x)=-M^n\) we have
\[
	(\M f)'(x)
=
-
\sqrt\varepsilon
\begin{cases}
1+\ord(\sqrt\varepsilon+1/M)
&
n\tx{ odd}
,\\
\ord(1/M)
&
n\tx{ even}
.
\end{cases}
\]
\end{corollary}

\begin{lemma}
\label{cor_mfppsuperlevelset}
For any odd \(1\leq n\leq N-1\) we have
\[
\lm{\{x>0:-2M^n< a(x)<-M^n\}}
\geq
(1+\ord(\sqrt\varepsilon+1/M))
\f{\sqrt\varepsilon}M
\]
and for every \(x>0\) with \(-2M^n< a(x)<-M^n\) we have
\[
	-(\M f)''(x)
\geq
(1/8+\ord(\sqrt\varepsilon+1/M))
M
.
\]
\end{lemma}

\begin{proof}
By \cref{lem_examfp} for every \(x>0\) with \(-2M^n\leq a(x)\leq -M^n\) we have
\[
	(\M f)'(x)
=
(1+\ord(\sqrt\varepsilon+1/M))
\f{
M^n \sqrt\varepsilon
}{
a(x)
}
,
\]
and so by \cref{eq_da} we obtain
\[
-a'(x)
=
	\f{-(\M f)'(x)}{f'(a(x))}
=
(1+\ord(\sqrt\varepsilon+1/M))
\f{
M^{2n+1}
}{
-a(x)\sqrt\varepsilon
}
.
\]
We can conclude the first part,
\begin{align*}
&
\lm{\{x>0:-2M^n<a(x)<-M^n\}}
\\
&\qquad\qquad=
\int_{-2M^n}^{-M^n}
-
(a^{-1})'
\geq
M^n
\inf_{-2M^n<a(x)<-M^n}
\f1{-a'(x)}
\\
&\qquad\qquad\geq
(1+\ord(\sqrt\varepsilon+1/M))
\f{\sqrt\varepsilon}M
.
\end{align*}
By \cref{lem_exax} we have
\[
x
=
(1+\ord(\sqrt\varepsilon+1/M))M^n\sqrt\varepsilon 
\]
and thus by \cref{eq_ddmf} we also obtain the second part,
\begin{align*}
	-(\M f)''(x)
&=
\f
{
	(2-a'(x))(\M f)'(x)
-
f'(x)
}{x-a(x)}
\\
&=
\f{
(1+\ord(\sqrt\varepsilon+1/M))M^{3n+1}/a(x)^2
+1
}{
-(1+\ord(\sqrt\varepsilon))a(x)
}
\\
&=
(1+\ord(\sqrt\varepsilon+1/M))
\f{
M^{3n+1}
}{
-a(x)^3
}
\\
&\geq
(1+\ord(\sqrt\varepsilon+1/M))
\f{
M
}{
8
}
.
\end{align*}
\end{proof}

\begin{proof}[Proof of \cref{theo_counterexample}]
By definition of \(f\) for \(1\leq q<\infty\) we have
\begin{align}
\nonumber
\var^q(f')
&=
\Bigl[
(1+\varepsilon/M)^q
+
\bigl(
(
1+\ord(1/M)
)
\varepsilon
\bigr)^q
(1+\ord(1/N))N
\Bigr]^{1/q}
\\
\label{eq_varqf}
&=
\Bigl[
1
+
\ord\bigl(
1/M
+
1/(N^{1/q}\varepsilon)
\bigr)
\Bigr]
N^{1/q}
\varepsilon
\end{align}
and by \cref{corollary_dmfvalues,eq_ordpower} we have
\begin{align*}
\var^q((\M f)')
&\geq
\bigl(
(1+\ord(\sqrt\varepsilon+1/M))
(1+\ord(1/N))
N
\varepsilon^{q/2}
\bigr)^{1/q}
\\
&=
(1+\ord(\sqrt\varepsilon+1/M+1/N))
N^{1/q}
\sqrt\varepsilon
.
\end{align*}
We can conclude
\[
\f{
	\var^q((\M f)')
}{
\var^q(f')
}
\geq
\f{
1+\ord(\sqrt\varepsilon+1/M+1/(N^{1/q}\varepsilon))
}{
\sqrt\varepsilon
}
,
\]
which tends to \(\infty\) for \(M\) and \(N^{1/q}\varepsilon\) large enough and \(\varepsilon\rightarrow0\).

By \cref{cor_mfppsuperlevelset} for \(M\) large and \(\varepsilon\) small enough we have
\begin{align*}
&
\lm{\{x>0:-(\M f)''(x)\geq M/9\}}
\\
&\qquad\qquad\geq
\bigcup_{k=1}^{\lfloor N/2\rfloor}
\lm{\{x>0:-2M^{2k-1}<a(x)<-M^{2k-1}\}}
\\
&\qquad\qquad=
(1+\ord(\sqrt\varepsilon+1/M)+\ord(1/N))
\f{
N\sqrt\varepsilon
}{
2M
}
\end{align*}
and with \cref{eq_varqf} we can conclude
\[
	\f M9
	\f{
		\lm{\{x>0:-(\M f)''(x)\geq M/9\}}
	}{
		\var(f')
	}
	\geq
	\f{
	1+\ord(\sqrt\varepsilon+1/M+1/(N\varepsilon))
	}{
	18\sqrt\varepsilon
	}
	,
\]
which tends to \(\infty\) for \(M\) and \(N\varepsilon\) large enough and \(\varepsilon\rightarrow0\).
\end{proof}

\printbibliography

\end{document}

%% file: preamble.tex
\usepackage{amsmath,amssymb,amsthm,enumitem,hyperref,url,xcolor}
\usepackage[capitalize]{cleveref}
\usepackage[style=alphabetic]{biblatex}

\setlist[enumerate,1]{label=(\roman*)}

\numberwithin{equation}{section}

\newtheorem{theorem}{Theorem}[section]

\newtheorem{lemma}[theorem]{Lemma}
\newtheorem{proposition}[theorem]{Proposition}
\newtheorem{corollary}[theorem]{Corollary}
\newtheorem{question}[theorem]{Question}
\theoremstyle{definition}

\newtheorem{remark}[theorem]{Remark}

\crefname{enumi}{}{}
\crefname{equation}{}{}
\Crefname{equation}{Formula}{Formulas}
\crefname{question}{Question}{Questions}
\Crefname{question}{Question}{Questions}

\newcommand\f\frac
\newcommand\tx\text
\newcommand\intd{\mathop{}\!\mathrm{d}}
\newcommand\seq{:=}

\newcommand\lm[1]{\mathcal L(#1)}

\DeclareMathOperator{\var}{var}
\DeclareMathOperator{\sign}{sign}
\DeclareMathOperator{\esssup}{ess\,sup}
\DeclareMathOperator{\essinf}{ess\,inf}
\newcommand\ind[1]{1_{#1}}
\newcommand\M{\mathrm M}
\newcommand\Mc{\mathrm M^{\text c}}
\newcommand\ord{\mathcal O}
\newcommand\loc{\mathrm{loc}}